\documentclass[11pt,twoside,letterpaper]{article} 
\usepackage{times,fancyhdr}
\usepackage{amsmath,
                    graphicx,amssymb}
\usepackage{graphicx}
\usepackage{hyperref}
\input xy
\xyoption{all}
\usepackage{xcolor}
\usepackage{amsthm}


\makeatletter
\def\cleardoublepage{\clearpage\if@twoside \ifodd\c@page\else%
    \hbox{}%
    \thispagestyle{empty}%
    \newpage%
    \if@twocolumn\hbox{}\newpage\fi\fi\fi}
\makeatother
\newtheorem{theorem}{Theorem}

\newtheorem{definition}{Definition}

\begin{document}
\title{
{\begin{flushleft}
\vskip 0.45in
{\normalsize\bfseries\textit{Chapter~1}}
\end{flushleft}
\vskip 0.45in
\bfseries\scshape Biunits of ternary algebra of hypermatrices}}
\author{\bfseries\itshape Viktor Abramov\thanks{E-mail address: viktor.abramov@ut.ee}\\
Institute of Mathematics and Statistics, University of Tartu, Estonia}
\date{}
\maketitle
\thispagestyle{empty}
\setcounter{page}{141}
\thispagestyle{fancy}
\fancyhead{}
\fancyhead[L]{In: Energy, Mass, Spin in Relativity, Gravitation, Cosmology \\
Editor: Editor Name, pp. {\thepage-\pageref{lastpage-01}}} 
\fancyhead[R]{ISBN 0000000000  \\
\copyright~2006 Nova Science Publishers, Inc.}
\fancyfoot{}
\renewcommand{\headrulewidth}{0pt}

\begin{abstract}
We study a ternary algebra of third-order hypermatrices, where by hypermatrix we mean a complex-valued quantity with three indices $T_{ijk}$. Ternary multiplications of hypermatrices have the property of generalized associativity. We introduce the concepts of $q$-cyclic and $\bar q$-cyclic traceless hypermatrix, where $q$ is a primitive third-order root of unity and study the structures of the corresponding subspaces. We show that $q$-cyclic traceless hypermatrices can be used to construct right biunits of the ternary algebra of third-order hypermatrices. We show the connection between the Clifford algebra structure induced by the quadratic invariant of $q$-cyclic traceless hypermatrices and ternary multiplication of hypermatrices. The motivation for studying the space of $q$-cyclic traceless hypermatrices is the ternary generalization of the Pauli exclusion principle.
\end{abstract}



\vspace{.08in} \noindent {\textbf AMS Subject Classification: 17A40, 20N10, 53C07} .

\pagestyle{fancy}
\fancyhead{}
\fancyhead[EC]{Viktor Abramov}
\fancyhead[EL,OR]{\thepage}
\fancyhead[OC]{Ternary Algebra of Hypermatrices}
\fancyfoot{}
\renewcommand\headrulewidth{0.5pt}
\label{lastpage-01}
\section{Introduction}
In this chapter, the main object of study is a ternary algebra of hypermatrices. By hypermatrix we mean a quantity with three indices $T_{ijk}$, where every index is an integer from $1$ to $n$. It should be noted that in this case there is no firmly established terminology, despite the fact that the use of hypermatrices in various fields such as algebra, geometry, theoretical physics and cybernetics is becoming increasingly popular. Quantities with three indices are also called 3-dimensional matrices, spatial matrices or three-index matrices. If we assume the tensor nature of the transformation of such a quantity with three indices, then we get a concept that also has several names, such as trivalent tensor, tensor of rank three or trilinear tensor. In this chapter we will use the terminology proposed in \cite{Abramov:Zapata_Arsiwalla_Beynon_2024}, where the authors advocate the need to establish a unified terminology in this area. Thus, we will call a quantity with three indices $T_{ijk}$ a hypermatrix. A tensor nature of such a quantity is also important for us. A tensor nature will be taken into account by assuming a naturally defined action of the rotation group $\mbox{SO}(3)$ on the vector space of hypermatrices.

A vector space of hypermatrices is of interest to abstract algebra primarily because it is a natural, so to speak, building material for ternary multiplication laws. It should be noted that in theoretical physics we can regularly see interest towards algebras with a ternary multiplication law. Here we note only few approaches. In 1980s there was proposed (\cite{Abramov:Bars_Gunaydin_1980} and references therein) an approach to fundamental constituents of matter based on ternary algebras which are building blocks of Lie algebras and superalgebras. In the mid-70s of the last century, Nambu proposed a generalization of Hamiltonian mechanics \cite{Abramov:Nambu_1973}, based on a ternary analogue of the Poisson bracket, called the Nambu-Poisson bracket. This approach was later developed in \cite{Abramov:Takhtajan_1994}. In the mid-1980s, Filippov proposed a generalization of Lie algebra based on an $n$-ary Lie bracket and the Filippov-Jacobi identity. Later, a surge of interest in the above generalizations was due to their applications in M-brane theory \cite{Abramov:Bagger_Lambert_2007}, \cite{Abramov:Baggert_Lamber_2008}, \cite{Abramov:Cherkis_Samann_2008}.

The physical motivation for the study of ternary algebras of hypermatrices proposed in this chapter is approach to a generalization of the Dirac operator and a ternary generalization of the Pauli exclusion principle proposed by R. Kerner \cite{Abramov:Abramov_Kerner_LeRoy_J_Math_Phys_1997}, \cite{Abramov:Kerner_1991}, \cite{Abramov:Kerner_2017}. We will briefly explain the ideas that underlay a ternary generalization of the Pauli exclusion principle. First of all, we note that the motivation for a ternary generalization of the Pauli exclusion principle is the properties of the quark model. Recall that quarks are the most fundamental particles of matter currently known. This means that all hadrons are composite particles made of quarks. Moreover, a fermionic baryon is a combination of three quarks, and a meson with integer spin is a pair quark-antiquark. Quarks have several quantum characteristics. One of the quantum characteristics is the flavor of a quark, and in each generation of quarks, and there are three of them, there are two flavors $(u,d), (c,s), (t,b)$. In addition to flavor, each quark has a color, electric charge, isospin, baryon number and other quantum characteristics. Note that all allowed combinations of quarks are colorless. In quantum chromodynamics, quarks are treated as fermions with half-integer spin and in this respect they are subject to the Pauli exclusion principle. In analogy with proton and neutron considered as two isospin components of a nucleon doublet the $u$-quark and $d$-quark (first generation) can be considered as two states of more general object. According to quantum chromodynamics, in a stable bound state there is place for two quarks in the same $u$-state or $d$-state, but not for three. This suggests a possible generalization of the Pauli exclusion principle, which can be formulated as follows: no three quarks with identically equal quantum characteristics can form a stable configuration that is perceived as a strongly interacting particle.

It is well known that the Pauli exclusion principle leads to the skew-symmetry of a wave function of a quantum fermion system. Let us find algebraic properties of a wave function of a system of particles, for example quarks, which obeys the ternary generalization of Pauli exclusion principle. Consider a quantum system of three particles. Let $|1>,|2>,\ldots,|n>$ be quantum states of these particles and $\Psi$ be a wave function of this system. The value of a wave function on three basic states $|i>,|j>,|k>$ will be denoted by $\Psi_{ijk}$. Then according to the ternary generalization of the Pauli exclusion principle for any basic state $|i>$ we must have $\Psi_{iii}=0$. If we now consider the superposition $|\varrho>$ of three basic states $\xi|i>+\zeta|j>+\eta|k>$, where $\xi,\zeta,\eta$ are complex numbers, then, according to the ternary generalization of Pauli exclusion principle, a wave function must vanish on this superposition, that is, $\Psi_{\varrho\varrho\varrho}=0$. Making use of the linearity of a wave function we conclude that the equation $\Psi_{\varrho\varrho\varrho}=0$ holds if and only if for any three basic states $|i>,|j>,|k>$ a wave function $\Psi$ satisfies
\begin{equation}
\Psi_{ijk}+\Psi_{jki}+\Psi_{kij}+\Psi_{kji}+\Psi_{jik}+\Psi_{ikj}=0,
\label{Abramov:introduction_wave_function}
\end{equation}
that is, the sum of the values of a wave function on all permutations of any three basic states is equal to zero.

Equation (\ref{Abramov:introduction_wave_function}) has two solutions:
\begin{itemize}
\item A wave function $\Psi_{ijk}$ is totally skew-symmetric, that is, permuting any two basic states in $\Psi_{ijk}$ we get minus. For example, in the case of first two basic states we have $\Psi_{ijk}=-\Psi_{jik}$. It is clear that this solution is consistent with the Pauli exclusion principle, that is, Fermi-Dirac statistics. Note that in this case a wave function vanishes whenever there are two equal states among $|i>,|j>,|k>$.
    Hence a wave function $\Psi_{ijk}$ is traceless over any pair of basic states, that is,
    \begin{equation}\Psi_{iik}=\Psi_{iki}=\Psi_{kii}=0.\label{Abramov:introduction_traceless}\end{equation}
    Here and in what follows we use Einstein's convention of summation over repeated indices.
    Note that in the present case the sum of values of a wave function on cyclic permutations of basic states, that is $\Psi_{ijk}+\Psi_{jki}+\Psi_{kij}$, may be different from zero.
\item A wave function $\Psi_{ijk}$ has the property
    \begin{equation}\Psi_{ijk}+\Psi_{jki}+\Psi_{kij}=0.\label{Abramov:introduction_cyclic_equation}\end{equation}
    In this case, a wave function can be nonzero even if we have two identical basic states (with equal quantum characteristics) among $|i>,|j>,|k>$. Thus, this case differs from Fermi-Dirac statistics because a quantum system allows two (but not three!) identical states. Note that, in turn, equation (\ref{Abramov:introduction_cyclic_equation}) can be also solved if we assume the following properties of the wave function
    $$
    \Psi_{ijk}=q\,\Psi_{jki},\;\;\mbox{or}\;\;\;\Psi_{ijk}=\bar q\,\Psi_{jki},
    $$
    where $q=\exp(2\pi\,i/3)$ is the third-order root of unity. Generally it does not follow from the property (\ref{Abramov:introduction_cyclic_equation}) that a wave function is traceless over any pair of indices. Therefore, in this part we have a difference from the first solution, that is, the skew-symmetric case. To make the analogy with the skew-symmetric case stronger, we will require that, in addition to property (\ref{Abramov:introduction_cyclic_equation}), a wave function be traceless over any pair of indices, i.e. (\ref{Abramov:introduction_traceless}).
\end{itemize}
In this chapter, the main object of study is the space of third-order hypermatrices $\mathfrak T^3$. A group of rotations $\mbox{SO}(3)$ acts on this space and under this action the hypermatrices transform like third-order covariant tensors. The vector space $\mathfrak T^3$ can be equipped with a ternary multiplication of hypermatrices, which has the property of generalized associativity. In this regard, recall that a set $H$ is called a semi-heap if it is equipped with a ternary multiplication
$$
(a,b,c)\in H\times H\times H\to a\cdot b\cdot c\in H,
$$
which has the property of generalized associativity
\begin{equation}
(a\cdot b\cdot c)\cdot f\cdot g=a\cdot (f\cdot c\cdot b)\cdot g=a\cdot b\cdot (c\cdot f\cdot g).
\label{Abramov:introduction_generalized_associativity}
\end{equation}
A semi-heap $H$ is said to be a ternary algebra if $H$ is a vector space. It should be noted that when the bracket in (\ref{Abramov:introduction_generalized_associativity}) is moved from the leftmost position to the center, elements $b$ and $f$ are rearranged. An excellent overview of the theory of semiheaps is given in \cite{Abramov:Zapata_Arsiwalla_Beynon_2024} (references therein), where in particular it is noted that the concept of a semiheap was introduced by V.V. Wagner in connection with an algebraic approach to the set of transition functions associated to an atlas of a manifold. In this chapter we use the terminology proposed in \cite{Abramov:Zapata_Arsiwalla_Beynon_2024}. We will need a notion which generalizes a concept of neutral element to ternary multiplications. An element $e$ of a semiheap $H$ is said to be a right (left) biunit if for any $a\in H$ it is satisfies $a\cdot e\cdot e=a$ ($e\cdot e\cdot a=a$). If $e\in H$ is a right biunit as well as left biunit then it is referred to as a biunit of a semiheap $H$. A heap is a semiheap $H$ whose any element $a$ is a biunit, that is, for any $b\in H$ it holds $a\cdot a\cdot b=b\cdot a\cdot a=b$.

In order to define a ternary algebra structure on the vector space of third-order hypermatrices we will use the following theorem
\cite{Abramov:Abramov_Kerner_Liivapuu_Shitov_2009}
\begin{theorem}
Let $A,B,C$ be $N$th order complex hypermatrices. Then there are only four different triple products of $N$th order complex  hypermatrices which obey the generalized associativity (\ref{Abramov:introduction_generalized_associativity}). These are
\begin{enumerate}
\item[1)]
$(A\odot B\odot C)_{ijk} = A_{ilm}B_{nlm}C_{njk},\quad
      A\odot B\odot C\rightarrow
\xy <1cm,0cm>:
(1,0)*+{A} , (2,0)*+{B} , (3,0)*+{C} ,
(1.25,-0.2)*+{\bullet} , (1.45,-0.2)*+{\circ} , (1.65,-0.2)*+{\circ} ,
(2.25,-0.2)*+{\circ} , (2.45,-0.2)*+{\circ} , (2.65,-0.2)*+{\circ} ,
(3.25,-0.2)*+{\circ} , (3.45,-0.2)*+{\bullet} , (3.65,-0.2)*+{\bullet} ,
(1.45,-0.26);(2.45,-0.26)**\crv{(1.5, -0.5)&(1.95, -0.7)&(2.4, -0.5)} ,
(1.65,-0.26);(2.65,-0.26)**\crv{(1.7, -0.4)&(2.15, -0.6)&(2.6, -0.4)} ,
(2.25,-0.26);(3.25,-0.26)**\crv{(2.3, -0.5)&(2.75, -0.7)&(3.2, -0.5)}
\endxy$
\item[2)]
$(A\odot B\odot C)_{ijk} = A_{ilm}B_{nml}C_{njk},\quad
     A\odot B\odot C \rightarrow
\xy <1cm,0cm>:
(1,0)*+{A} , (2,0)*+{B} , (3,0)*+{C} ,
(1.25,-0.2)*+{\bullet} , (1.45,-0.2)*+{\circ} , (1.65,-0.2)*+{\circ} ,
(2.25,-0.2)*+{\circ} , (2.45,-0.2)*+{\circ} , (2.65,-0.2)*+{\circ} ,
(3.25,-0.2)*+{\circ} , (3.45,-0.2)*+{\bullet} , (3.65,-0.2)*+{\bullet} ,
(1.45,-0.26);(2.65,-0.26)**\crv{(1.5, -0.5)&(2.05, -0.7)&(2.6, -0.5)} ,
(1.65,-0.26);(2.45,-0.26)**\crv{(1.7, -0.4)&(2.05, -0.6)&(2.4, -0.4)} ,
(2.25,-0.26);(3.25,-0.26)**\crv{(2.3, -0.5)&(2.75, -0.7)&(3.2, -0.5)}
\endxy$
\item[3)]
$(A\odot B\odot C)_{ijk} = A_{ijl}B_{nml}C_{mnk},\quad
    A\odot B\odot C \rightarrow
\xy <1cm,0cm>:
(1,0)*+{A} , (2,0)*+{B} , (3,0)*+{C} ,
(1.25,-0.2)*+{\bullet} , (1.45,-0.2)*+{\bullet} , (1.65,-0.2)*+{\circ} ,
(2.25,-0.2)*+{\circ} , (2.45,-0.2)*+{\circ} , (2.65,-0.2)*+{\circ} ,
(3.25,-0.2)*+{\circ} , (3.45,-0.2)*+{\circ} , (3.65,-0.2)*+{\bullet} ,
(1.65,-0.26);(2.65,-0.26)**\crv{(1.7, -0.5)&(2.15, -0.7)&(2.6, -0.5)} ,
(2.25,-0.26);(3.45,-0.26)**\crv{(2.3, -0.5)&(2.85, -0.7)&(3.4, -0.5)} ,
(2.45,-0.26);(3.25,-0.26)**\crv{(2.45, -0.4)&(2.85, -0.6)&(3.2, -0.4)} ,
\endxy$
\item[4)]
$(A\odot B\odot C)_{ijk} = A_{ijl}B_{mnl}C_{mnk},\quad
      A\odot B\odot C\rightarrow
\xy <1cm,0cm>:
(1,0)*+{A} , (2,0)*+{B} , (3,0)*+{C} ,
(1.25,-0.2)*+{\bullet} , (1.45,-0.2)*+{\bullet} , (1.65,-0.2)*+{\circ} ,
(2.25,-0.2)*+{\circ} , (2.45,-0.2)*+{\circ} , (2.65,-0.2)*+{\circ} ,
(3.25,-0.2)*+{\circ} , (3.45,-0.2)*+{\circ} , (3.65,-0.2)*+{\bullet} ,
(1.65,-0.26);(2.65,-0.26)**\crv{(1.7, -0.5)&(2.15, -0.7)&(2.6, -0.5)} ,
(2.25,-0.26);(3.25,-0.26)**\crv{(2.3, -0.5)&(2.75, -0.7)&(3.2, -0.5)} ,
(2.45,-0.26);(3.45,-0.26)**\crv{(2.45, -0.4)&(2.95, -0.6)&(3.4, -0.4)} ,
\endxy$
\label{Abramov:introduction_theorem_ternary_multiplications}
\end{enumerate}
\end{theorem}
This theorem uses diagrams located at the right end of the lines 1), 2), 3), 4) to schematically describe associative ternary multiplications of hypermatrices. In these diagrams, black circles denote free indices (no summation), and summation (or contraction) shown by an arc connecting two white circles is performed over the indices corresponding to white circles. From this theorem it follows that the vector space of $N$th order hypermatrices endowed with one of the ternary multiplications 1 - 4 is a ternary algebra. Particularly $\mathfrak T^3$ is a ternary algebra. According to the terminology used in \cite{Abramov:Zapata_Arsiwalla_Beynon_2024}, the ternary multiplication of 4 is called a fish product. Indeed, ternary multiplication 4 can be schematically represented by the following diagram, which superficially resembles the silhouette of a fish (vertices marked in red indicate summation over the corresponding indices).
\begin{figure}[htp]
\centering
  \includegraphics[width=14cm]{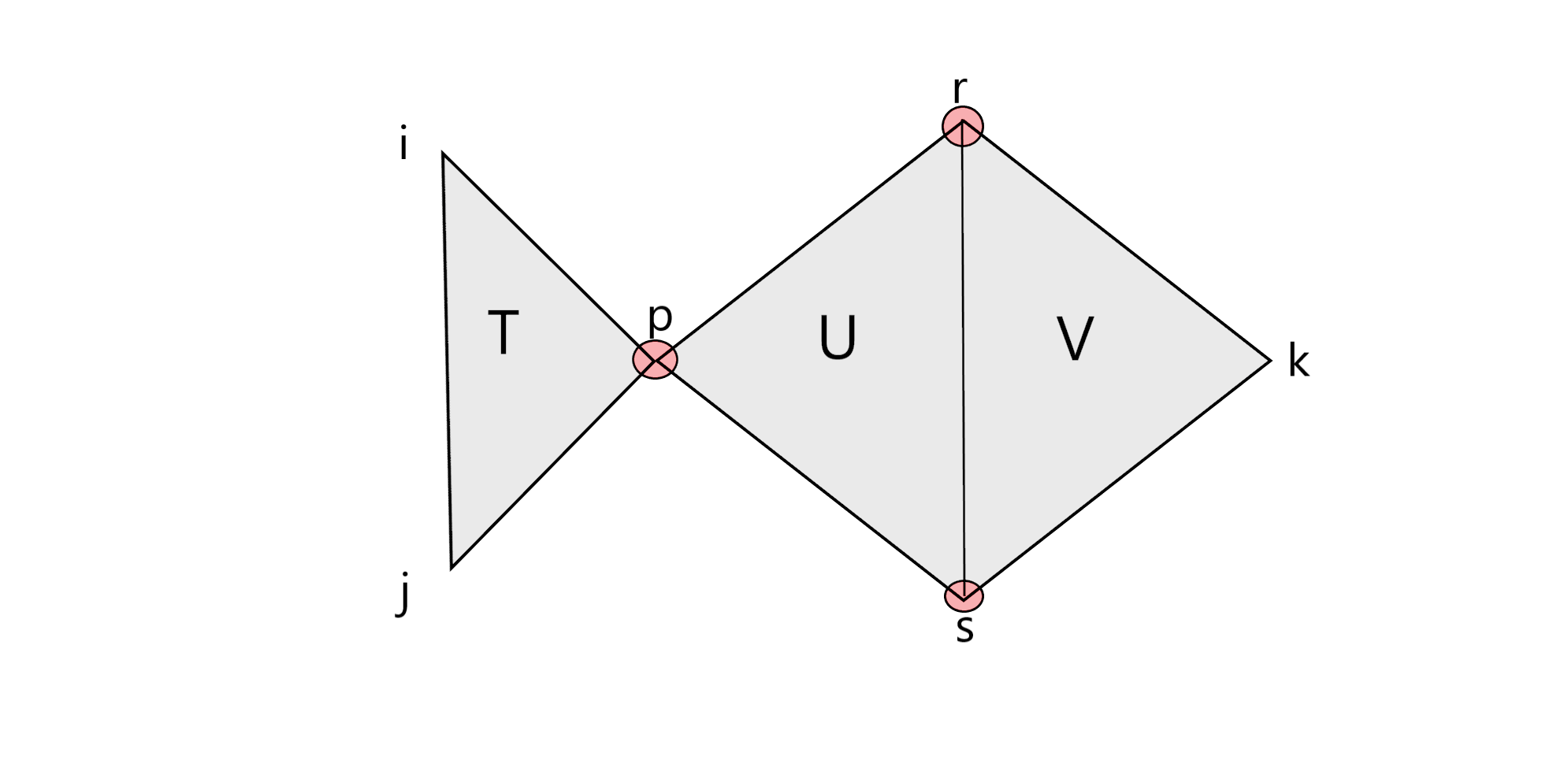}\\
  \caption{Ternary multiplication $(T\bullet U\bullet V)_{ijk} = T_{ijp}U_{rsp}V_{rsk}$}\label{Abramov:hypermatrix_figure}
\end{figure}
To differentiate between the ternary multiplications 3 and 4, we will denote them as follows
\begin{equation}
(T\diamond U\diamond V)_{ijk} = T_{ijp}U_{rsp}V_{srk},\;\;
               (T\bullet U\bullet V)_{ijk} = T_{ijp}U_{rsp}V_{rsk},
\label{Abramov:two_ternary_products_circ_bullet}
\end{equation}
where $T,U,V\in {\mathfrak T}^3$. The corresponding ternary algebras of hypermatrices will be denoted as pairs $({\mathfrak T}^3,\diamond), ({\mathfrak T}^3,\bullet)$. When a statement holds for both ternary multiplications, we will denote them with a single symbol $\odot$.

In the theory of tensor representations of the rotation group \cite{Abramov:Gelfand_Minlos_Shapiro_2018} it is shown that the space of third-order hypermatrices $\mathfrak T^3$ can be uniquely decomposed into a direct sum of subspaces $\mathfrak T^3_a, a=0,1,2,3$, where $a$ is a weight of representation. In each of the subspaces $\mathfrak T^3_a$ we have a representation of the rotation group multiple to an irreducible one. Of particular interest to us are subspaces $\mathfrak T^3_0$ and $\mathfrak T^3_2$. The direct sum of these spaces is the space of traceless hypermatrices and $\mathfrak T^3_0$ is the subspace of completely skew-symmetric hypermatrices, and $\mathfrak T^3_2$ is the subspace of traceless hypermatrices with the property
\begin{equation}
T_{ijk}+T_{jki}+T_{kij}=0.
\label{Abramov:introduction_cyclic_hypermatrix}
\end{equation}
In $\mathfrak T^3_0$ there is an irreducible representation of $\mbox{SO}(3)$ of the weight zero and in $\mathfrak T^3_2$ there is a twofold irreducible representation of $\mbox{SO}(3)$ of the weight 2. In order to split the twofold irreducible representation into two irreducible representations one can decompose the subspace $\mathfrak T^3_2$ into a direct sum of two subspaces by solving the equation (\ref{Abramov:introduction_cyclic_hypermatrix}) as it was shown above (in the case of a wave function)
\begin{equation}
T_{ijk}=q\,T_{jki},\;\;(q-\mbox{cyclic hypermatrix}),\;\;\;T_{ijk}=\bar q\,T_{jki},\;\;(\bar q-\mbox{cyclic hypermatrix}).
\end{equation}
In this chapter, we calculate the quadratic $\mbox{SO}(3)$-invariants of $q$-cyclic hypermatrices and the list of these invariants can be found in \cite{Abramov:Ahmad_2011}. We prove theorem which states that if $T$ is a $q$-cyclic hypermatrix such that $I_2\neq 0$, where $I_2$ is a quadratic $\mbox{SO}(3)$-invariant, then $T$ multiplied by an appropriate factor is a right biunit of the ternary algebra $({\mathfrak T}^3,\diamond)$. The quadratic invariant $I_2$ induces a non-degenerate bilinear form $K$ on the space of $q$-cyclic hypermatrices. This non-degenerate bilinear form induces a Clifford algebra structure on the complex 5-dimensional space of $q$-cyclic hypermatrices. We show the connection between Clifford algebra structure and ternary hypermatrix multiplication $T\diamond U\diamond V$.
\section{Ternary algebra of hypermatrices}
In this section we consider the vector space of complex third-order hypermatrices. This vector space will be denoted by $\mathfrak T^3$. The dimension of this vector space is 27. The rotation group $\mbox{SO}(3)$ acts on the vector space of complex third-order hypermatrices as follows
\begin{equation}
{\tilde T}_{prs}=g_{pi}\,g_{rj}\,g_{sk}\;T_{ijk},\;\;g=(g_{ij})\in\mbox{SO}(3),\;T,\tilde T\in {\mathfrak T^3}.
\label{Abramov:transformation_of_tensor}
\end{equation}
Here and henceforth we use the Einstein's convention of summation over twice repeated indexes. It is worth to mention that the formula (\ref{Abramov:transformation_of_tensor}) defines the tensor representation of the rotation group $\mbox{SO}(3)$ in the complex vector space $\frak T^3$.

Thus we can identify the components of a third-order hypermatrix $T$ with components of the third-order covariant tensor. {We} assume that the entries of a hypermatrix $T$ are located in 3-dimensional space in such a way that connecting them with straight line segments we get a cube.
\begin{eqnarray}
\xymatrix@!0{
& & \textcolor{red}{T_{311}}   \ar@{-}[rrr]\ar@{-}'[d]'[dd][ddd]
& & & \textcolor{red}{T_{312}} \ar@{-}[rrr]\ar@{-}'[d]'[dd][ddd]
& & & \textcolor{red}{T_{313}} \ar@{-}[ddd]
\\
& T_{211}     \ar@{-}[ur] \ar@{-}[rrr] \ar@{-}'[d][ddd]
& & & T_{212} \ar@{-}[ur] \ar@{-}[rrr] \ar@{-}'[d][ddd]
& & & T_{213} \ar@{-}[ur]              \ar@{-}[ddd]
\\
\textcolor{blue}{T_{111}}       \ar@{{-}}[ur] \ar@{-}[rrr] \ar@{-}[ddd]
& & & \textcolor{blue}{T_{112}} \ar@{-}[ur] \ar@{-}[rrr] \ar@{-}[ddd]
& & & \textcolor{blue}{T_{113}} \ar@{-}[ur] \ar@{-}[ddd] \ar@{-}[ddd]
\\
& & \textcolor{red}{T_{321}}   \ar@{-}'[r]'[rr][rrr] \ar@{-}'[d]'[dd][ddd]
& & & \textcolor{red}{T_{322}} \ar@{-}'[r]'[rr][rrr] \ar@{-}'[d]'[dd][ddd]
& & & \textcolor{red}{T_{323}} \ar@{-}[ddd]
\\
& T_{221}     \ar@{-}[ur] \ar@{-}'[rr][rrr] \ar@{-}'[d][ddd]
& & & T_{222} \ar@{-}[ur] \ar@{-}'[rr][rrr] \ar@{-}'[d][ddd]
& & & T_{223} \ar@{-}[ur] \ar@{-}[ddd]
\\
\textcolor{blue}{T_{121}}       \ar@{-}[ur] \ar@{-}[rrr] \ar@{-}[ddd]
& & & \textcolor{blue}{T_{122}} \ar@{-}[ur] \ar@{-}[rrr] \ar@{-}[ddd]
& & & \textcolor{blue}{T_{123}} \ar@{-}[ur]              \ar@{-}[ddd]
\label{Abramov:cube}\\
& & \textcolor{red}{T_{331}}   \ar@{-}'[r]'[rr][rrr]
& & & \textcolor{red}{T_{332}} \ar@{-}'[r]'[rr][rrr]
& & & \textcolor{red}{T_{333}}
\\
& T_{231}     \ar@{-}[ur] \ar@{-}'[rr][rrr]
& & & T_{232} \ar@{-}[ur] \ar@{-}'[rr][rrr]
& & & T_{233} \ar@{-}[ur]
\\
\textcolor{blue}{T_{131}}       \ar@{-}[rrr]\ar@{-}[ur]
& & & \textcolor{blue}{T_{132}} \ar@{-}[rrr]\ar@{-}[ur]
& & & \textcolor{blue}{T_{133}} \ar@{-}[ur]
}
\end{eqnarray}
We can define three directions in a hypermatrix $T$. The direction of increase of the index $i$ (first index) will be referred to as the $i$-direction of a hypermatrix $T$. Similarly we define the $j$-direction and $k$-direction of a hypermatrix. It is useful to split a hypermatrix $T$ into third-order square matrices. It is usually done by sections of the cube of a hypermatrix by planes orthogonal to edges of the cube or to three directions of a hypermatrix defined above. If we cut the cube (\ref{Abramov:cube}) with three planes perpendicular to the $i$-direction of hypermatrix we get three third-order square matrices corresponding to the values $i=1,2,3$, which we denote by $T^{(1)}_1=(T_{1jk}), T^{(1)}_2=(T_{2jk}), T^{(1)}_3=(T_{3jk})$ respectively. By other words the subscript $(1)$ shows that we fix the first subscript in $T_{ijk}$ and the subscript shows the value of the fixed subscript.
\begin{figure}[htp]
\centering
  \includegraphics[width=9cm]{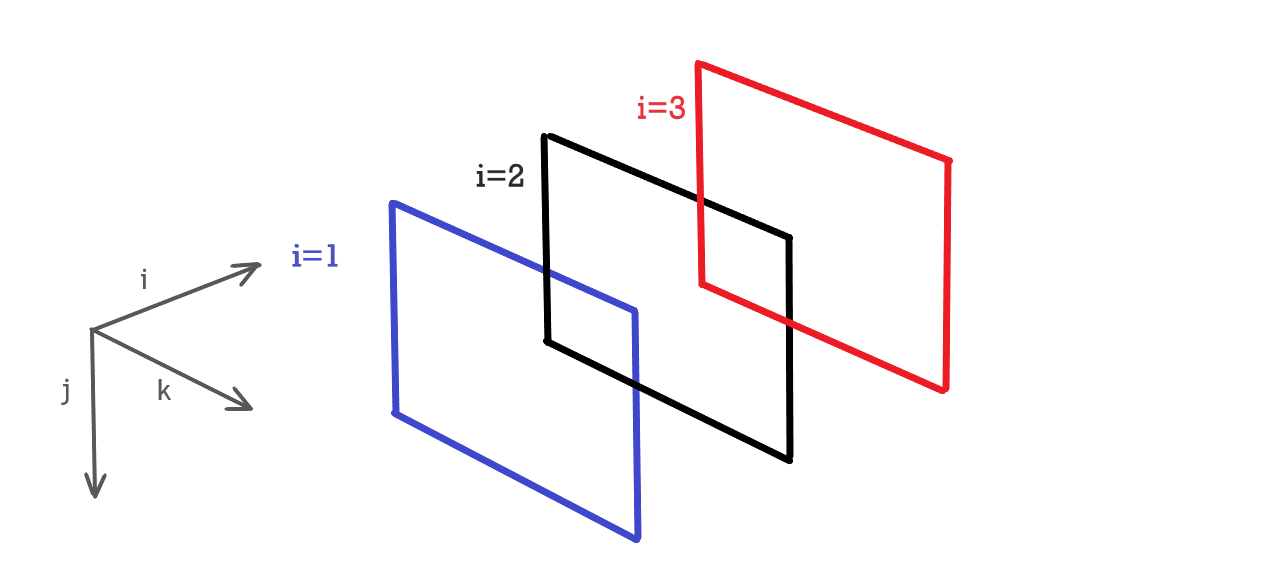}\\
  \caption{Sections of a hypermatrix perpendicular to $i$-direction}\label{Abramov:hypermatrix_figure}
\end{figure}
Arranging these three square matrices one after the other from left to right in the plane of the page and separating them by vertical lines, we obtain a useful notation for the hypermatrix $T$ in the form
\begin{equation}
T=\left(
  \begin{array}{ccc}
    T_{{1}11} & T_{{1}12} & T_{{1}13} \\
    T_{{1}21} & T_{{1}22} & T_{{1}23} \\
    T_{{1}31} & T_{{1}32} & T_{{1}33} \\
  \end{array}
\left|
\begin{array}{ccc}
    T_{{2}11} & T_{{2}12} & T_{{2}13} \\
    T_{{2}21} & T_{{2}22} & T_{{2}23} \\
    T_{{2}31} & T_{{2}32} & T_{{2}33} \\
  \end{array}
\right|
  \begin{array}{ccc}
    T_{{3}11} & T_{{3}12} & T_{{3}13} \\
    T_{{3}21} & T_{{3}22} & T_{{3}23} \\
    T_{{3}31} & T_{{3}32} & T_{{3}33} \\
  \end{array}
\right).
\label{Abramov:i-direction}
\end{equation}
Analogously fixing the value of the second subscript $j$ (the third subscript $k$) in $T_{ijk}$ we get the square matrices $T^{(2)}_1,T^{(2)}_2,T^{(2)}_3$ ($T^{(k)}_{1},T^{(k)}_{2},T^{(k)}_{3}$). For example, in order to write out the entries of the matrix $T^{(3)}_3$ using hypermatrix (\ref{Abramov:i-direction}) we must take the third column in each square matrix $T^{(1)}_1, T^{(1)}_2, T^{(1)}_3$, compose them into a square matrix and then to transpose it.

Such important operations with square matrices as determinant and trace can be extended to the space of hypermatrices by constructing corresponding analogues \cite{Abramov:Sokolov_1960}. In this chapter we will be interested in an algebraic structure based on a ternary multiplication of hypermatrices. Let us remind that a semiheap is a set $H$ equipped with a ternary multiplication $(a,b,c)\in H\times H\times H\to a\cdot b\cdot c\in H$ which for any five elements $a,b,c,d,u,v\in H$ satisfies the generalized associativity
\begin{equation}
(a\cdot b\cdot c)\cdot u\cdot v=a\cdot (u\cdot c\cdot b)\cdot v=a\cdot b\cdot (c\cdot u\cdot v).
\label{Abramov:generalized_associativity}
\end{equation}
In this chapter we will study the case when a set $H$ has the structure of a vector space. In this case, a semiheap $H$ is called a ternary algebra. Thus, ternary algebra is a vector space equipped with an associative (\ref{Abramov:generalized_associativity}) ternary multiplication law. A wide class of ternary algebras can be constructed using hypermatrices. 

In particular, the complex vector space of third-order tensors $\mathfrak T^3$, equipped with one of the ternary multiplications 1 - 4 shown in Theorem \ref{Abramov:introduction_theorem_ternary_multiplications} is a ternary algebra and this algebra is the main object of study in this chapter. Since the structure of ternary multiplications 1,2 is similar to the structures of ternary multiplications 3,4, we will study ternary algebras of hypermatrices with ternary multiplications 3,4.
\section{Invariants and irreducible representations}
In this chapter we will need invariants of the third-order hypermatrices under the action of the rotation group. The complete set of linear and quadratic $\mbox{SO}(3)$-invariants is given in \cite{Abramov:Ahmad_2011}. For third-order hypermatrix $T$ there is only one linear $\mbox{SO}(3)$-invariant
$$
I=\epsilon_{ijk}T_{ijk}=T_{123}+T_{231}+T_{312}-T_{321}-T_{213}-T_{132}.
$$
The complete set of quadratic $\mbox{SO}(3)$-invariants of a complex hypermatrix $T=(T_{ijk})$ consists of
\begin{eqnarray}
I_1 &=& T_{ijk}T_{ijk},\;I_1^\ast=T_{ijk}\overline T_{ijk},\;I_2=T_{ijk}T_{ikj},\;I^\ast_2=T_{ijk}\overline T_{ikj},\nonumber\\
I_3 &=& T_{ijk}\,T_{jik},\;I^\ast_3=T_{ijk}\,\overline{T}_{jik},\;I_4=T_{ijk}\,T_{kji},\;I^\ast_4=T_{ijk}\overline{T}_{kji},\nonumber\\
I_5 &=& T_{ijk}T_{kij}+T_{ijk}T_{jki},\;I_5^\ast=T_{ijk}\overline{T}_{kij}+T_{ijk}\overline{T}_{jki},\nonumber\\
I_{6} &=& T_{iik}T_{ppk},\;I^\ast_{6} = T_{iik}\overline T_{ppk},\;I_{7} = T_{iji}T_{pjp},\;I^\ast_{7} = T_{iji}\overline T_{pjp},\label{Abramov:invariants}\\
I_{8} &=& T_{ijj}T_{iqq},\;I^\ast_{8} = T_{ijj}\overline T_{iqq},\;I_{9} = T_{iik}T_{kqq},\;I^\ast_{9} = T_{iik}\overline T_{kqq},\nonumber\\
I_{10} &=& \frac{1}{2} (T_{iik}T_{pkp}+T_{iji}T_{ppj}),\;I^\ast_{10} =\frac{1}{2} (T_{iik}\overline T_{pkp}+T_{iji}\overline T_{ppj}),\nonumber\\
I_{11} &=& \frac{1}{2} (T_{iji}T_{jqq}+T_{ijj}T_{pip}),\;I^\ast_{11} = \frac{1}{2} (T_{iji}\overline T_{jqq}+T_{ijj}\overline T_{pip}).\nonumber
\end{eqnarray}

Recall that the rotation group acts on the vector space of hypermatrices $\mathfrak T^3$ and this action encodes the tensor nature of the hypermatrices. Thus, any structure on the vector space of hypermatrices must be consistent with the action of the rotation group, that is, be invariant with respect to this action. Let us denote this action as follows
\begin{equation}
(g,T)\in \mbox{SO}(3)\times {\mathfrak T}^3\to g\cdot T\in{\mathfrak T}^3.
\label{Abramov:action_of_rotation_group}
\end{equation}
Then (\ref{Abramov:transformation_of_tensor}) can be written as
$$
(g\cdot T)_{prs}=g_{pi}\,g_{rj}\,g_{sk}\;T_{ijk}.
$$
Since the ternary multiplications (\ref{Abramov:two_ternary_products_circ_bullet}) are constructed by means of contractions of tensors over a pair of indices, it is obvious that the ternary product of three tensors $T,U,V$ is an $\mbox{SO}(3)$-tensor. Hence the ternary multiplications (\ref{Abramov:two_ternary_products_circ_bullet}) commute with the action of the rotation group (or its tensor representation), that is, we have
$$
(g\cdot T)\odot (g\cdot U)\odot (g\cdot V)=g\cdot (T\odot U\odot V).
$$

Another $\mbox{SO}(3)$-invariant structure on the ternary algebra $\mathfrak T^3$ is a Hermitian metric. We use the invariant $I^\ast_1$ (\ref{Abramov:invariants}) to construct a Hermitian metric on the ternary algebra ${\mathfrak T}^{3}$. Hence we define the Hermitian scalar product of two hypermatrices $T,U\in{\mathfrak T}^3$ as follows
\begin{equation}
h(T,U)=T_{ijk}\,\overline U_{ijk}.
\label{Abramov:Hermitian_metric}
\end{equation}
This way of defining the Hermitian metric $h$ is natural, since in this case we consider the entries of a hypermatrix as the coordinates of the 27-dimensional complex  vector. Evidently
$$
h(g\cdot T,g\cdot U)=h(T,U).
$$

In the theory of representations of the rotation group \cite{Abramov:Gelfand_Minlos_Shapiro_2018}, it is shown that the representation of the rotation group in the complex vector space of hypermatrices of third order (\ref{Abramov:action_of_rotation_group}) can be uniquely decomposed into irreducible representations as follows
$$
{\mathfrak T}^3=\oplus_{w=0}^3\;{\mathfrak T}^3_w,
$$
where $w$ is a weight of representation, $\mathfrak T^3_0$ is the vector space of totally skew-symmetric third-order hypermatrices, $\mathfrak T^3_1$ is the vector space of third-order hypermatrices of the type
$$
T_{ijk}=\delta_{ij} T_k+\delta_{ik} T_j+\delta_{jk} T_i,
$$
$\mathfrak T^3_2$ is the vector space of traceless hypermatrices (trace over any pair of indices is zero) satisfying the equation
$$
T_{ijk}+T_{jki}+T_{kij}=0,
$$
and $\mathfrak T^3_3$ is the vector space of traceless totally symmetric hypermatrices.
Thus
\begin{eqnarray}
{\mathfrak T}^3_0 &=& \{T\in{\mathfrak T}^3: T_{ijk}=\lambda\,\epsilon_{ijk}, \lambda\in{\mathbb C}\},\nonumber\\
{\mathfrak T}^3_1 &=& \{T\in{\mathfrak T}^3: T_{ijk}=\delta_{ij} T_k+\delta_{ik} T_j+\delta_{jk} T_i\},\nonumber\\
{\mathfrak T}^3_2 &=& \{T\in{\mathfrak T}^3: T_{iik}=0,\,T_{iji}=0,\,T_{ijj}=0,\,T_{ijk}+T_{jki}+T_{kij}=0\},\nonumber\\
{\mathfrak T}^3_3 &=& \{T\in{\mathfrak T}^3: T_{iik}=0,\,T_{iji}=0,\,T_{ijj}=0, T_{i_{\sigma(1)}i_{\sigma(2)}i_{\sigma(3)}}=T_{i_1i_2i_3},\,\forall \sigma\in S_3\},\nonumber
\end{eqnarray}
where $\epsilon_{ijk}$ is the Levi-Civita symbol and $S_3$ is the group of substitutions of the set $\{1,2,3\}$. These subspaces have the following dimensions
$$
\mbox{dim}\,{\mathfrak T}^3_0=1,\;\;\mbox{dim}\,{\mathfrak T}^3_1=9,\;\;\mbox{dim}\,{\mathfrak T}^3_2=10,\;\;\mbox{dim}\,{\mathfrak T}^3_3=7.
$$

In each of these subspaces we have a representation of the rotation group that is a multiple of an irreducible one. In the subspace ${\mathfrak T}^3_0$ we have only one irreducible representation (of weight 0), in the subspace ${\mathfrak T}^3_1$ a three times repeated (threefold) irreducible representation (of weight 1), in the subspace ${\mathfrak T}^3_2$ a twice repeated (twofold) irreducible representation (of weight 2) and in the subspace ${\mathfrak T}^3_3$ we have one irreducible representation (of weight 3). It is worth to note that every subspace ${\mathfrak T}^3_w$ is invariant under the action of the rotation group. It is also worth noting that the direct sum of three subspaces ${\mathfrak T}^3_0, {\mathfrak T}^3_2,{\mathfrak T}^3_3$  is the subspace of traceless third-order hypermatrices (trace over any pair of subscripts is zero), which we denote by ${\mathfrak T}^3_\mathtt{tr}$. Hence ${\mathfrak T}^3={\mathfrak T}^3_1\oplus {\mathfrak T}^3_\mathtt{tr}$, where $\mbox{dim}\,{\mathfrak T}^3_\mathtt{tr}=18$.

In order to decompose a representation that is a multiple of an irreducible representation into irreducible representations, we must decompose the representation space into a direct sum of subspaces (their number is equal to the multiplicity of a representation) so that this decomposition is invariant under the action of the rotation group. In this chapter, we will study the subspace of the weight 2 representation, that is, the subspace ${\mathfrak T}^3_2$. This is because the subspace ${\mathfrak T}^3_2$ is most closely related to the ternary generalization of the Pauli principle. In this subspace there is a twofold irreducible representation of the rotation group. Hence in order to split this twofold irreducible representation into two irreducible representations we have to split the representation space ${\mathfrak T}^3_2$ into a direct sum of two subspaces invariant with respect to the action of the rotation group. For this purpose we will use a substitution operator. Let
$
\varsigma\in S_3
$
be the cyclic substitution $\varsigma(1)=2,\varsigma(2)=3,\varsigma(3)=1$. Define the substitution operator $\mathtt{L}_{\varsigma}:{\mathfrak T}^3\to {\mathfrak T}^3$ as follows $\mathtt{L}_{\varsigma}(T)_{i_1i_2i_3}=T_{i_{\varsigma(1)}i_{\varsigma(2)}i_{\varsigma(3)}}=T_{i_2i_3i_1}$. Obviously the substitution operator $\mathtt{L}_{\varsigma}$ is a linear operator and $\mathtt{L}^3_{\varsigma}=\mathtt{Id}$, where $\mathtt{Id}$ is the identity operator. The latter implies that the substitution operator ${\mathtt L}_{\varsigma}$ has three eigenvalues $1,q,\bar q$, where $q=\exp{(2i\pi/3)}$ is the primitive third order root of unity and $\bar q$ is its complex conjugate. It is well known that $1+q+\bar q=0$. Hence we can decompose the vector space of third-order hypermatrices ${\mathfrak T}^3$ into the direct sum of three subspaces ${\mathfrak T}^{3,1},{\mathfrak T}^{3,q},{\mathfrak T}^{3,\bar q}$ corresponding to the eigenvalues $1, q, \bar q$ respectively. Hence
\begin{equation}
{\mathfrak T}^3={\mathfrak T}^{3,1}\oplus{\mathfrak T}^{3,q}\oplus{\mathfrak T}^{3,\bar q},
\label{Abramov:decomposition_into_three_subspaces}
\end{equation}
where
\begin{eqnarray}
{\mathfrak T}^{3,1}&=&\{T\in{\mathfrak T}^3: \mathtt{L}_{\varsigma}(T)=T\},\nonumber\\
{\mathfrak T}^{3,q}&=&\{T\in{\mathfrak T}^3: \mathtt{L}_{\varsigma}(T)=q\,T\},\nonumber\\
{\mathfrak T}^{3,\bar q} &=& \{T\in{\mathfrak T}^3: \mathtt{L}_{\varsigma}(T)=\bar q\,T\}.\nonumber
\nonumber
\end{eqnarray}
The decomposition (\ref{Abramov:decomposition_into_three_subspaces}) can be obtained from more elementary considerations. It is easy to see that any third-order hypermatrix $T$ can be represented in the form
\begin{eqnarray}
T_{ijk} = \frac{1}{3}(T_{ijk}+T_{jki}+T_{kij})+\frac{1}{3}(T_{ijk}+{\bar q}\,T_{jki}+q\,T_{kij})+\frac{1}{3}(T_{ijk}+{q}\,T_{jki}+{\bar q}\,T_{kij}).\nonumber
\end{eqnarray}
In order to give this decomposition a more concise form we introduce the following square polynomials of the substitution operator
\begin{eqnarray}
\xi_1 &=& \frac{1}{3}(\mathtt{Id}+\mathtt{L}_{\varsigma}+\mathtt{L}^2_{\varsigma}),\nonumber\\
     \xi_q &=& \frac{1}{3} (\mathtt{Id}+\bar q\,\mathtt{L}_{\varsigma}+q\,\mathtt{L}^2_{\varsigma}),\nonumber\\
         \xi_{\bar q} &=& \frac{1}{3} (\mathtt{Id}+q\,\mathtt{L}_{\varsigma}+\bar q\,\mathtt{L}^2_{\varsigma}).\nonumber
\end{eqnarray}
Thus, the operators $\xi_1,\xi_q,\xi_{\bar q}$ are operators of projection of the hypermatrix onto subspaces ${\mathfrak T}^{3,1},{\mathfrak T}^{3,q},{\mathfrak T}^{3,\bar q}$  respectively. It is easy to verify the following properties
$$
\xi_1+\xi_q+\xi_{\bar q}=\mathtt{Id},\;{\mathtt L}_\xi\,\xi_1=\xi_1\,{\mathtt L}_\xi=\xi_1,\,
        {\mathtt L}_\xi\,\xi_q=\xi_q\,{\mathtt L}_\xi=q\,\xi_q, {\mathtt L}_\xi\,\xi_{\bar q}=\xi_{\bar q}\,{\mathtt L}_\xi\,=\bar q\,\xi_{\bar q}.
$$
Now we can describe the weight 2 representation vector space ${\mathfrak T}^3_2$ as follows
$$
{\mathfrak T}^3_2=\mbox{Ker}\,\xi_1\cap {\mathfrak T}^3_\mathtt{tr}.
$$
It is easy to verify that the subspace ${\mathfrak T}^3_2$ is invariant with respect to the substitution operator $\mathtt{L}_{\varsigma}$, that is, the restriction of the operator $\mathtt{L}_{\varsigma}$ to the subspace ${\mathfrak T}^3_2$ is a correctly defined linear operator on ${\mathfrak T}^3_2$, i.e. $\mathtt{L}_{\varsigma}:{\mathfrak T}^3_2\to {\mathfrak T}^3_2$. Assume that $T\in {\mathfrak T}^3_\mathtt{tr}$. Then $\mathtt{L}_{\varsigma}(T)\in {\mathfrak T}^3_\mathtt{tr}.$ Indeed
$$
\mathtt{L}_{\varsigma}(T)_{iij}=T_{iji}=0,
$$
and analogously for the other traces. From the property ${\mathtt L}_\xi\,\xi_1=\xi_1\,{\mathtt L}_\xi=\xi_1$ it follows that the subspace $\mbox{Ker}\,\xi_1$ is invariant under the action of the operator ${\mathtt L}_\xi$.

Thus the 10-dimensional weight 2 representation vector space splits into two 5-dimensional subspaces ${\mathfrak T}^3_2={\mathfrak T}^{3,q}_2\oplus {\mathfrak T}^{3,\bar q}_2$, where ${\mathfrak T}^{3,q}_2={\mathfrak T}^{3,q}\cap {\mathfrak T}^{3}_2$ and ${\mathfrak T}^{3,\bar q}_2={\mathfrak T}^{3,\bar q}\cap {\mathfrak T}^{3}_2$. It is worth to note that ${\mathfrak T}^{3,1}\cap {\mathfrak T}^{3}_2=\{0\}$. Thus in each 5-dimensional subspace ${\mathfrak T}^{3,q}_2,{\mathfrak T}^{3,\bar q}_2$ we have an irreducible representation of the rotation group. Since the resulting subspaces ${\mathfrak T}^{3,q}_2,{\mathfrak T}^{3,\bar q}_2$ are very important for what follows, it is useful to give them an explicit description
\begin{eqnarray}
{\mathfrak T}^{3,q}_2 &=& \{T\in{\mathfrak T}^3: T_{iij}=0,\,T_{iji}=0,\,T_{jii}=0,\,T_{ijk}={\bar q}\,T_{jki}\},\label{Abramov:bar_q_hypermatrices}\\
{\mathfrak T}^{3,\bar q}_2 &=& \{T\in{\mathfrak T}^3: T_{iij}=0,\,T_{iji}=0,\,T_{jii}=0,\,T_{ijk}={q}\,T_{jki}\}.\label{Abramov:q_hypermatrices}
\end{eqnarray}
The notations ${\mathfrak T}^{3,q}_2$ and ${\mathfrak T}^{3,\bar q}_2$ are quite complex. Unfortunately, we have not found simpler notations that would include all the information about these subspaces. Formulas (\ref{Abramov:bar_q_hypermatrices}) and (\ref{Abramov:q_hypermatrices}) show that in fact the hypermatrices of these subspaces are determined by only two conditions, where one is that they are traceless, and the second is a law of transformation of hypermatrix entries under a cyclic permutation of subscripts. Therefore, to simplify the presentation, we will call hypermatrices of (\ref{Abramov:bar_q_hypermatrices}) traceless $\bar q$-cyclic hypermatrices, meaning by the latter that $T_{ijk}={\bar q}\,T_{jki}$. Note that the same formula can be represented in the form $T_{jki}={q}\,T_{ijk}$ that clearly shows that such hypermatrices are eigenvectors of the substitution operator ${\mathtt L}_\varsigma$ with the eigenvalue $q$. This is the reason why we use $q$ in the notation ${\mathfrak T}^{3,q}_2$ for the subspace of this kind of hypermatrices. Analogously the hypermatrices of (\ref{Abramov:q_hypermatrices}) will be referred to as traceless $q$-cyclic hypermatrices, where $q$-cyclic stands for $T_{ijk}={q}\,T_{jki}$.

It is easy to see that complex conjugation maps subspace ${\mathfrak T}^{3,q}_2$ to subspace ${\mathfrak T}^{3,\bar q}_2$ and vice versa. Indeed, if a traceless hypermatrix $T=(T_{ijk})$ transforms under cyclic permutation $\varsigma$ according to the formula $T_{ijk}=\bar q\,T_{jki}$, i.e. $T\in {\mathfrak T}^{3,q}_2$, then its complex conjugate $\overline T=(\overline T_{ijk})$ will transform under the same cyclic permutation according to the formula $\overline T_{ijk}=\overline{\bar q\,T_{jki}}=q\,\overline T_{jki}$, i.e. $\overline T\in {\mathfrak T}^{3,\bar q}_2$.
\section{The space of traceless $q$-cyclic hypermatrices}
Now our aim is to study the structure of the 5-dimensional complex vector space of traceless $q$-cyclic hypermatrices ${\mathfrak T}^{3,\bar q}_2.$ In this space we have an irreducible representation of the rotation group $\mbox{SO}(3)$. We choose the following five traceless $q$-cyclic hypermatrices as a basis
\begin{eqnarray}
E_1 &=& \frac{1}{\sqrt{6}}\left(
  \begin{array}{ccc}
    0\! &\! 0 \!&\!\! 0 \!\\[0.2cm]
    0\! &\! 1 \!&\!\! 0 \!\\[0.2cm]
    0\! &\! 0 \!&\!\! -1 \!\\
  \end{array}
\left|
\begin{array}{ccc}
 0\! &\! q \!&\! 0\! \\[0.2cm]
 {\bar q} \!&\! 0 \!&\! 0\!\\[0.2cm]
    0 \!&\! 0 \!&\! 0 \!\\
  \end{array}
\right|
  \begin{array}{ccc}
  \!0 \!&\! 0 \!&\!\! -q \!\\[0.2cm]
  \!0 \!&\! 0 \!&\!\! 0 \!\\[0.2cm]
 \!-{\bar q} \!&\!\! 0 \!&\! 0 \!\\
  \end{array}
\right),\nonumber
\end{eqnarray}
\begin{eqnarray}
E_2 &=& \frac{1}{\sqrt{6}}\left(
  \begin{array}{ccc}
   \! 0 \!&\!\! -{\bar q} \!&\! 0 \!\\[0.2cm]
   \! -q \!&\!\! 0 \!&\! 0\\[0.2cm]
   \! 0 \!&\!\! 0 \!&\! 0 \!\\
  \end{array}
\left|
\begin{array}{ccc}
 \!-1 \!&\! 0 \!&\! 0 \!\\[0.2cm]
\! 0 \!&\! 0 \!&\! 0 \!\\[0.2cm]
  \!  0 \!&\! 0 \!&\! 1 \!\\
  \end{array}
\right|
  \begin{array}{ccc}
  0 \!&\! 0 \!&\! 0 \!\\[0.2cm]
  0 \!&\! 0 \!&\! q\\[0.2cm]
 0 \!&\! {\bar q} \!&\! 0 \!\\
  \end{array}
\right),\nonumber\\
E_3 &=& \frac{1}{\sqrt{6}}\left(
  \begin{array}{ccc}
    0\! &\! 0 \!&\!\! {\bar q} \!\\[0.2cm]
    0\! &\! 0 \!&\!\! 0 \!\\[0.2cm]
    q\! &\! 0 \!&\!\! 0 \!\\
  \end{array}
\left|
\begin{array}{ccc}
 0\! &\! 0 \!&\! 0\! \\[0.2cm]
 0 \!&\! 0 \!&\! -{\bar q}\!\\[0.2cm]
    0 \!&\! -q \!&\! 0 \!\\
  \end{array}
\right|
  \begin{array}{ccc}
  \!1 \!&\! 0 \!&\!\! 0 \!\\[0.2cm]
  \!0 \!&\! -1 \!&\!\! 0 \!\\[0.2cm]
 \! 0 \!&\!\! 0 \!&\! 0 \!\\
  \end{array}
\right),\label{Abramov:basis}\\
E_4 &=& \frac{1}{\sqrt{3}}\left(
  \begin{array}{ccc}
   \! 0 \!&\!\! 0 \!&\! 0 \!\\[0.2cm]
   \! 0 \!&\!\! 0 \!&\! 1\\[0.2cm]
   \! 0 \!&\!\! 0 \!&\! 0 \!\\
  \end{array}
\left|
\begin{array}{ccc}
 \!0 \!&\! 0 \!&\! 0 \!\\[0.2cm]
\! 0 \!&\! 0 \!&\! 0 \!\\[0.2cm]
  \!  {\bar q} \!&\! 0 \!&\! 0 \!\\
  \end{array}
\right|
  \begin{array}{ccc}
  0 \!&\! q \!&\! 0 \!\\[0.2cm]
  0 \!&\! 0 \!&\! 0\\[0.2cm]
 0 \!&\! 0 \!&\! 0 \!\\
  \end{array}
\right),\nonumber\\
E_5 &=&\frac{1}{\sqrt{3}}
\left(
  \begin{array}{ccc}
   \! 0 \!&\!\! 0 \!&\! 0 \!\\[0.2cm]
   \! 0 \!&\!\! 0 \!&\! 0\\[0.2cm]
   \! 0 \!&\!\! q \!&\! 0 \!\\
  \end{array}
\left|
\begin{array}{ccc}
 \!0 \!&\! 0 \!&\! {\bar q} \!\\[0.2cm]
\! 0 \!&\! 0 \!&\! 0 \!\\[0.2cm]
  \!  0 \!&\! 0 \!&\! 0 \!\\
  \end{array}
\right|
  \begin{array}{ccc}
  0 \!&\! 0 \!&\! 0 \!\\[0.2cm]
  1 \!&\! 0 \!&\! 0\\[0.2cm]
 0 \!&\! 0 \!&\! 0 \!\\
  \end{array}
\right).\nonumber
\end{eqnarray}
It is easy to verify that this basis is an orthonormal basis with respect to the Hermitian metric $h$ (\ref{Abramov:Hermitian_metric}), that is, $h(E_A,E_B)=\delta_{AB},$ where the subscripts denoted by capital Latin letters run from 1 to 5. Using the orthonormal basis $\{E_A\}$, we can identify a vector $z=(z^A)$ of the 5-dimensional complex vector space $\mathbb C^5$ with the third-order hypermatrix $T(z)$ as follows
$$
z=(z^A)\in {\mathbb C}^5\mapsto T(z)=z^A\,E_A\in {\mathfrak T}^{3,{\bar q}}_2,
$$
where
\begin{equation}
T(z)=\left(
  \begin{array}{ccc}
    0 & -\frac{{\bar q}\,z^2}{\sqrt{6}} & \frac{{\bar q}\,z^3}{\sqrt{6}} \\[0.2cm]
    -\frac{q\,z^2}{\sqrt{6}} & \frac{z^1}{\sqrt{6}} & \frac{\,z^4}{\sqrt{3}}\\[0.2cm]
    \frac{q\,z^3}{\sqrt{6}} & \frac{q\,z^5}{\sqrt{3}} & -\frac{z^1}{\sqrt{6}} \\
  \end{array}
\left|
\begin{array}{ccc}
 -\frac{z^2}{\sqrt{6}} & \frac{q\,z^1}{\sqrt{6}} & \frac{{\bar q}\,z^5}{\sqrt{3}} \\[0.2cm]
    \frac{{\bar q}\,z^1}{\sqrt{6}} & 0 & -\frac{{\bar q}\,z^3}{\sqrt{6}}\\[0.2cm]
    \frac{{\bar q}\,z^4}{\sqrt{3}} & -\frac{q\,z^3}{\sqrt{6}} & \frac{z^2}{\sqrt{6}} \\
  \end{array}
\right|
  \begin{array}{ccc}
  \frac{z^3}{\sqrt{6}} & \frac{q\,z^4}{\sqrt{3}} & -\frac{q\,z^1}{\sqrt{6}} \\[0.2cm]
    \frac{z^5}{\sqrt{3}} & -\frac{z^3}{\sqrt{6}} & \frac{q\,z^2}{\sqrt{6}}\\[0.2cm]
    -\frac{{\bar q}\,z^1}{\sqrt{6}} & \frac{{\bar q}\,z^2}{\sqrt{6}} & 0 \\
  \end{array}
\right).
\nonumber
\end{equation}
Throughout what follows we will keep in mind that the 5-dimensional complex vector space $\mathbb C^5$ is identified with the subspace of third-order hypermatrices ${\mathfrak T}^{3,{\bar q}}_2$. In other words, we consider a 5-dimensional complex (Hermitian) space, each point of which is a third-order hypermatrix (three-dimensional matrix), that is, it has the shape of a cube (\ref{Abramov:cube}).

Since the subspace ${\mathfrak T}^{3,{\bar q}}_2$ is invariant under the action of the rotation group $\mbox{SO}(3)$, the set of invariants (\ref{Abramov:invariants}), when restricted to the subspace ${\mathfrak T}^{3,{\bar q}}_2$, will give the set of subspace ${\mathfrak T}^{3,{\bar q}}_2$ invariants. Calculating these invariants in coordinates $z^A$, we get
\begin{eqnarray}
&&\!\!\!\!\!\! I_1=0,\qquad\qquad\qquad\qquad\quad I^{\ast}_1=\sum_{A=1}^5\;z^A\,{\bar z}^A=h(z,\bar z),\\
&&\!\!\!\!\!\! I_2=\sum_{A=1}^3\,(z^A)^2+2\,q\,z^4z^5,\;\;\;
        I^{\ast}_2=0,\label{form:I_2}\\
&&\!\!\!\!\!\! I_3=q\;I_2,
  \;\;\;\;\;\;\;\qquad\qquad\qquad I^\ast_3={\bar q}\,I^{\ast}_2=0,\\
&&\!\!\!\!\!\! I_4={\bar q}\;I_2,\;\;\;\;\;\;\;\;\;\;\qquad\qquad\quad I^\ast_4=q\;I^\ast_2=0,\\
&&\!\!\!\!\!\! I_5=0,\;\;\;\;\;\;\;\;\;\;\;\;\;\;\;\;\;\qquad\qquad
       I_5^\ast=-I^\ast_1,
\end{eqnarray}
and the invariants $I_6,I^\ast_6,I_7,I^\ast_7,I_8,I^\ast_8,I_9,I^\ast_9,I_{10},I^\ast_{10},I_{11},I^\ast_{11}$ vanish because they are constructed by means of the trace of a hypermatrix (hypermatrices in ${\mathfrak T}^{3,{\bar q}}_2$ are traceless). Thus we have two independent invariants in the subspace ${\mathfrak T}^{3,{\bar q}}_2$, where on is the canonical Hermitian metric $h(z,\bar z)$ and the other is the quadratic form $(z^1)^2+(z^2)^2+(z^3)^2+2q\,z^4z^5$, which will be denoted by $K(z,z).$

The quadratic invariant $K(z,z)$ will play an important role in what follows. In paper \cite{Abramov:Abramov_Liivapuu_2024}  the properties of this invariant were studied. Let $K(z,z)=K_{AB}z^Az^B$, where $K_{AB}$ is the 5th order matrix of the quadratic form $K(z,z)$
$$
(K_{AB})=\left(
                     \begin{array}{ccccc}
                1 & 0 & 0 & 0 & 0 \\
                0 & 1 & 0 & 0 & 0 \\
                0 & 0 & 1 & 0 & 0 \\
                0 & 0 & 0 & 0 & q \\
                0 & 0 & 0 & q & 0 \\
                     \end{array}\right).
$$
The matrix $(K_{AB})$ can also be considered as a twice covariant tensor in the 5-dimensional complex vector space $\mathbb C^5$, that is,
\begin{equation}
{\tilde K}_{AB}={\mathtt U}^C_A\,{\mathtt U}^D_B\;K_{CD},
\label{form:transformation rule}
\end{equation}
where $({\mathtt U}^C_A)$ is a 5th order unitary matrix. For any orthonormal basis $\{E_A\}$ the second-order covariant tensor
$K_{AB}=K(E_A,E_B)$ determined by the quadratic form  $K(z,z)$ has the following properties which are unitary invariant:
\begin{itemize}
\item $K_{AB}=K_{BA}$ (symmetric),
\item $K_{AB}\,\overline{K}_{CB}=\delta_{AB}$ (unitary),
\item $\mbox{det}\,(K_{AB})=\epsilon$, where $\epsilon=e^{i\pi/3}$ is the sixth order root of unity.
\end{itemize}
It also has the following properties, which are invariant with respect to real unitary transformations:
\begin{itemize}
\item $K^6=E$, where the tensor $K=(K_{AB})$ is considered as a matrix,
\item the eigenvalues of the tensor $K=(K_{AB})$ are $1,1,1,q,-q$.
\end{itemize}
\begin{definition}
A third-order hypermatrix $T=(T_{ijk})$ is said to be $I_2$-regular if the second $\mbox{SO}(3)$-invariant $I_2(T)=T_{ijk}T_{ikj}$ is non-zero.
\end{definition}
Obviously if $T$ is a third-order traceless $q$-cyclic hypermatrix and $z^A$ are the coordinates of this hypermatrix in the orthonormal basis (\ref{Abramov:basis}) then $T$ is $I_2$-regular if $K(z,z)\neq 0$.
\section{Right biunits of ternary algebra of hypermatrices}
Now our aim is to show that each $I_2$-regular traceless $q$-cyclic hypermatrix $U$ can be used to construct a right biunit for the ternary product $T\diamond V\diamond W$. To this end, we will derive a more general formula for the ternary product $T\diamond T(u)\diamond T(v)$, where $T$ is an arbitrary third-order hypermatrix $T\in {\mathfrak T}^3$ and $T(u), T(v)$ are traceless $q$-cyclic hypermatrices with coordinates $(u^A),(v^B)$ respectively. It is convenient to split the calculation of the ternary product into two stages. Let us remind that
$$
(T\diamond U\diamond V)_{ijk}=T_{ijp}\,U_{rsp}\,V_{srk}.
$$
It is easy to see that $U,V$ part of this product can be considered as the trace of the product of two square matrices $U^{(3)}_p=(U_{msp}), V^{(3)}_k=(V_{snk})$. Thus it is useful to introduce an auxiliary third-order square matrix $H$ as follows
$$
H=(H_{pk}),\;\;\;H_{pk}=\mbox{Tr}\,(U^{(3)}_p\,V^{(3)}_k).
$$
Now the calculation of the entries of the ternary product $T\diamond U\diamond V$ can be performed by sequentially taking the rows of three square matrices in (\ref{Abramov:i-direction}) and multiplying them on the right by the auxiliary matrix $H$. The three entries obtained in this way will be the row of the ternary product, that is,
\begin{equation}
(T\diamond U\diamond V)_{ijk}=T_{ijp}\,H_{pk}.
\end{equation}
First of all, we write down the square matrices $U^{(3)}_p,V^{(3)}_k$. As we have already mentioned, for this we must take three columns of square matrices in (\ref{Abramov:i-direction}) (the first, second and third columns, respectively), form a matrix from them and transpose it. We obtain
\begin{eqnarray}
U^{(3)}_1 \!\!\!\!&=&\!\!\!\!\frac{1}{\sqrt{6}}\left(
            \begin{array}{ccc}
              0 & -q\,u^2 & q\,u^3 \\
              -u^2 & \bar q\,u^1 & \sqrt{2}\,\bar q\,u^4 \\
              u^3 & \sqrt{2}\,u^5 & -\bar q\,u^1 \\
            \end{array}
          \right),\;\;\;\;                 V^{(3)}_1 \!\!=\!\!\frac{1}{\sqrt{6}}\left(
                                                                        \begin{array}{ccc}
                                                                        0 & -q\,v^2 & q\,v^3 \\
                                                                       -v^2 & \bar q\,v^1 & \sqrt{2}\,\bar q\,v^4 \\
                                                                       v^3 & \sqrt{2}\,v^5 & -\bar q\,v^1 \\
                                                                         \end{array}
                                                                                \right),\nonumber\\
U^{(3)}_2 \!\!\!\!&=&\!\!\!\!\frac{1}{\sqrt{6}}\left(
            \begin{array}{ccc}
              -\bar q\,u^2 & u^1 & \sqrt{2}\,q\,u^5 \\
              q\,u^1 & 0 & -q\,u^3 \\
              \sqrt{2}\,q\,u^4 & -u^3 & \bar q\,u^2 \\
            \end{array}
          \right), \;\;               V^{(3)}_2 \!\!=\!\!\frac{1}{\sqrt{6}}\left(
                                                                      \begin{array}{ccc}
                                                                      -\bar q\,v^2 & v^1 & \sqrt{2}\,q\,v^5 \\
                                                                      q\,v^1 & 0 & -q\,v^3 \\
                                                                      \sqrt{2}\,q\,v^4 & -v^3 & \bar q\,v^2 \\
                                                                       \end{array}
                                                                              \right),\nonumber\\
U^{(3)}_3 \!\!\!\!&=&\!\!\!\!\frac{1}{\sqrt{6}}\left(
            \begin{array}{ccc}
              \bar q\,u^3 & \sqrt{2}\,u^4 & -u^1 \\
              \sqrt{2}\,\bar q\,u^5 & -\bar q\,u^3 & u^2 \\
              -q\,u^1 & q\,u^2 & 0 \\
            \end{array}
          \right), \;\;\;\;                V^{(3)}_3 \!\!=\!\!\frac{1}{\sqrt{6}}\left(
                                                                       \begin{array}{ccc}
                                                                      \bar q\,v^3 & \sqrt{2}\,v^4 & -v^1 \\
                                                                      \sqrt{2}\,\bar q\,v^5 & -\bar q\,v^3 & v^2 \\
                                                                      -q\,v^1 & q\,v^2 & 0 \\
                                                                        \end{array}
                                                                                \right).\nonumber
\end{eqnarray}
Then taking the trace of the products of these matrices we find the entries of the auxiliary matrix $H$. The diagonal entries of $H$ are all equal and multiple of the invariant bilinear form $K(u,v)$, i.e.
\begin{eqnarray}
H_{pp} = \frac{q}{3}(u^1v^1+u^2v^2+u^3v^3+q\,u^4v^5+q\,u^5v^4)=\frac{q}{3}\,K(u,v),\;\;p=1,2,3.\label{Abramov:symmetric part}
\end{eqnarray}
Non-diagonal entries $H_{pk},p\neq k$ can be written as follows
\begin{eqnarray}
H_{12} &=& \frac{q}{6}\left|
                         \begin{array}{cc}
                           u^2 & u^1 \\
                           v^2 & v^1 \\
                         \end{array}
                       \right|+\frac{\bar q}{3\sqrt{2}}\left|
                         \begin{array}{cc}
                           u^3 & u^4 \\
                           v^3 & v^4 \\
                         \end{array}
                       \right|+\frac{q}{3\sqrt{2}}\left|
                         \begin{array}{cc}
                           u^3 & u^5 \\
                           v^3 & v^5 \\
                         \end{array}
                       \right|,\nonumber\\
H_{23} &=& \frac{q}{6}\left|
                         \begin{array}{cc}
                           u^3 & u^2 \\
                           v^3 & v^2 \\
                         \end{array}
                       \right|+\frac{q}{3\sqrt{2}}\left|
                         \begin{array}{cc}
                           u^1 & u^4 \\
                           v^1 & v^4 \\
                         \end{array}
                       \right|+\frac{\bar q}{3\sqrt{2}}\left|
                         \begin{array}{cc}
                           u^1 & u^5 \\
                           v^1 & v^5 \\
                         \end{array}
                       \right|,\nonumber\\
H_{31} &=& \frac{q}{6}\left|
                         \begin{array}{cc}
                           u^1 & u^3 \\
                           v^1 & v^3 \\
                         \end{array}
                       \right|+\frac{1}{3\sqrt{2}}\left|
                         \begin{array}{cc}
                           u^2 & u^4 \\
                           v^2 & v^4 \\
                         \end{array}
                       \right|+\frac{1}{3\sqrt{2}}\left|
                         \begin{array}{cc}
                           u^2 & u^5 \\
                           v^2 & v^5 \\
                         \end{array}
                       \right|,\nonumber
\end{eqnarray}
and the entries $H_{21}, H_{32}, H_{13}$ are obtained by simultaneously rearranging the columns in each of the second order determinants in $H_{12}, H_{23}, H_{31}$  respectively. Hence $H_{pk}=-H_{kp}$. The non-diagonal elements of the matrix $H_{pk}$ can be written in compact form if we introduce the following third-order hypermatrix $\tau$, which can be considered as a $q$-analog of the Levi-Civita symbol. We define $\tau$ as follows
\begin{eqnarray}
\tau_{ijk}=\begin{cases}
                0,\; \mbox{if there are at least two equal subscripts among}\;i,j,k\\
                \tau_{ijk}=q\,\tau_{jki}\;\mbox{for any even permutation of integers}\;1,2,3\\
                \tau_{ijk}=\bar q\,\tau_{jki}\;\mbox{for any odd permutation of integers}\;1,2,3\\
                \tau_{312}=1,\;\tau_{132}=-1.
           \end{cases}
\end{eqnarray}
From this definition we can easily find all the entries of the hypermatrix $\tau$
\begin{eqnarray}
\tau_{123} \!\!\!&=&\!\!\! \bar q,\;\tau_{231}=q,\;\tau_{312}=1,\nonumber\\
\tau_{213} \!\!\!&=&\!\!\! -\bar q,\;\tau_{321}=-q,\;\tau_{132}=-1.\nonumber
\end{eqnarray}
Thus, the structure of hypermatrix $\tau$ with respect to permutations of subscripts is in some way a mixture of a $q$-cyclic structure and a $\bar q$-cyclic structure, that is, by cyclically permuting an even permutation of integers 1,2,3, the entries of the hypermatrix $\tau$ are transformed according to the $q$-cyclic law, and in the case of an odd permutation we have a $\bar q$-cyclic law. In addition, the hypermatrix $\tau$ is skew-symmetric in the first two subscripts, that is, $\tau_{ijk}=-\tau_{jik}$. Moreover, if $i,j,k$ is an even permutation of $1,2,3$ then $\tau_{kji}=-{\bar q}\tau_{ijk}, \tau_{ikj}=-q\,\tau_{ijk}$.
Now making use of the hypermatrix $\tau$ we can write the non-diagonal entries of the matrix $H$ as follows
\begin{equation}
H_{pk}=\frac{q}{6}\,\left|
                         \begin{array}{cc}
                           u^k & u^p \\
                           v^k & v^p \\
                         \end{array}
                       \right|+\frac{1}{3\sqrt{2}}\tau_{pkr}\left|
                         \begin{array}{cc}
                           u^r & u^4 \\
                           v^r & v^4 \\
                         \end{array}
                       \right|+\frac{1}{3\sqrt{2}}\bar\tau_{pkr}\left|
                         \begin{array}{cc}
                           u^r & u^5 \\
                           v^r & v^5 \\
                         \end{array}
                       \right|
\end{equation}
Thus, these calculations show that the diagonal elements of matrix $H$ form its symmetric part, which we will denote by $H_{\mathtt{symm}}$, and the non-diagonal elements form its skew-symmetric part, which we will denote by $H_{\mathtt{skew}}$. Then $H=H_{\mathtt{symm}}+H_{\mathtt{skew}}$. It is easy to see that the symmetric part of matrix $H$ is determined by the bilinear form $K(u,v)$, while the skew-symmetric part is determined by the bivector constructed by means of two vectors $u,v$ of the 5-dimensional complex space. It should be noted that the bilinear form $K(u,v)$ is non-degenerate and therefore it determines the structure of the Clifford algebra on the 5-dimensional space of traceless $q$-cyclic hypermatrices.
\begin{theorem}
Let $U$ be a third-order traceless $q$-cyclic $I_2$-regular hypermatrix. Then the hypermatrix
$$
\hat U=\sqrt{\frac{3}{q I_2(U)}}\;U,
$$
is a right biunit of the ternary algebra $({\mathfrak T}^3,\diamond)$, that is, for any third-order hypermatrix $T$ we have $T\diamond\hat U\diamond\hat U=T$.
\end{theorem}
\begin{proof}
Let $T$ be a third-order hypermatrix, $U,V$ be two third-order traceless $q$-cyclic hypermatrices whose coordinates in the basis (\ref{Abramov:basis}) are $u^A,v^B$ respectively. Then
$$
(T\diamond U\diamond V)_{ijk}=T_{ijp}\,H_{pk},
$$
where
$$
H_{pk}=\frac{q}{6}\,\left|
                         \begin{array}{cc}
                           u^k & u^p \\
                           v^k & v^p \\
                         \end{array}
                       \right|+\frac{1}{3\sqrt{2}}\tau_{pkr}\left|
                         \begin{array}{cc}
                           u^r & u^4 \\
                           v^r & v^4 \\
                         \end{array}
                       \right|+\frac{1}{3\sqrt{2}}\bar\tau_{pkr}\left|
                         \begin{array}{cc}
                           u^r & u^5 \\
                           v^r & v^5 \\
                         \end{array}
                       \right|
$$
Now if we assume $U=V$, that is $u^A=v^A$, then the skew-symmetric part of the auxiliary matrix $H$ vanishes and we are left with the symmetric part
$$
H_{pk}=\frac{q}{3}\,K(u,v)\,\delta_{pk}.
$$
Hence
$$
(T\diamond U\diamond U)_{ijk}=T_{ijp}\,\frac{q}{3}\,K(u,u)\,\delta_{pk}=\frac{q}{3}\,K(u,u)\,T_{ijk},
$$
or
$$
T\diamond U\diamond U=\frac{qI_2(U)}{3}\,T,
$$
where we used $K(u,u)=I_2(U)$ and $I_2(U)$ is the value of the second invariant $I_2$ calculated in the case of a hypermatrix $U$. Since $U$ is $I_2$-regular hypermatrix, that is, $I_2(U)\neq 0$, we can consider the hypermatrix
$$
\hat U=\sqrt{\frac{3}{qI_2(U)}}\;U,
$$
which clearly satisfies
$$
T\diamond \hat U\diamond \hat U=T\diamond(\sqrt{\frac{3}{qI_2(U)}}\;U)\diamond(\sqrt{\frac{3}{qI_2(U)}}\;U)=
                     \frac{3}{qI_2(U)}\;T\diamond U\diamond U=T
$$

\end{proof}

\end{document}